\documentclass[11pt,leqno]{article}
\usepackage{verbatim}
\usepackage[english]{babel}
\usepackage{enumitem}
\usepackage{tikz}
\usepackage{graphicx}
\usepackage{booktabs,tabularx}
\usepackage{verbatim}
\usepackage{etoolbox}
\usepackage{tkz-euclide}
\usepackage{geometry}
\usepackage{mathtools}
\usepackage{color}
\usepackage[title]{appendix}
\usepackage[all]{xy}
\usepackage{tikz-cd}
\usepackage{blindtext}
\usepackage[T1]{fontenc}
\usepackage{amsthm}
\usepackage{amsfonts}
\usepackage{txfonts}
\usepackage{palatino,amssymb,epsfig}
\usepackage{latexsym,epsf,epic,amscd}
\usepackage{mathrsfs}
\usepackage{graphicx}
\usepackage{caption}
\usepackage{subcaption}

\usepackage{appendix}
\usepackage{amsmath}
\usepackage{bookmark} 
\usepackage[nottoc]{tocbibind}
\hypersetup{
	colorlinks=true,
	linkcolor=black,
	citecolor=black,
	anchorcolor=black,
	urlcolor=black}
\allowdisplaybreaks[4]
\numberwithin{equation}{section}
\DeclareFontFamily{OMX}{yhex}{}
\DeclareFontShape{OMX}{yhex}{m}{n}{<->yhcmex10}{}
\DeclareSymbolFont{yhlargesymbols}{OMX}{yhex}{m}{n}
\DeclareMathAccent{\wideparen}{\mathord}{yhlargesymbols}{"F3}

\DeclareMathOperator{\area}{Area}

\newcommand{\Addresses}{{
		\bigskip
		\footnotesize
	
		\noindent Guangming Hu, \href{18810692738@163.com}{18810692738@163.com}
		\newline\textit{ College of Science, Nanjing University of Posts and Telecommunications,
			Nanjing, 210003, P.R. China.}\par\nopagebreak
		\medskip
		\noindent Ziping Lei, \href{zplei@ruc.edu.cn}{zplei@ruc.edu.cn}
		\newline\textit{ School of Mathematics, Renmin University of China, Beijing, 100872, P.R. China.} \par\nopagebreak
		\medskip
	\noindent Yanlin Li, \href{liyl@hznu.edu.cn}{liyl@hznu.edu.cn}
		\newline\textit{ School of Mathematics, Hangzhou Normal University, Hangzhou, 311121, P.R. China.} \par\nopagebreak
		\medskip
		\noindent Hao Yu, \href{b453@cnu.edu.cn}{b453@cnu.edu.cn}
		\newline\textit{ Academy for multidisciplinary studies, Capital Normal University, Beijing, 100048, P.R.China} }}

\title{ Boundary Value Problem and Discrete Schwarz-Pick Lemma for Generalized Hyperbolic Circle Packings}
\author{ Guangming Hu, Ziping Lei, Yanlin Li and Hao Yu}

\date{}
\newtheorem{theorem}{Theorem}[section]
\newtheorem{lemma}[theorem]{Lemma}
\newtheorem{proposition}[theorem]{Proposition}
\newtheorem{corollary}[theorem]{Corollary}
\theoremstyle{definition}
\newtheorem{definition}[theorem]{Definition}
\newtheorem{remark}[theorem]{Remark}

\newcommand{\pp}[2]{\frac{\partial#1}{\partial#2}}

\begin{document}
	\maketitle
	
	\begin{abstract}
		
	In 1991,  Beardon and Stephenson \cite{be} generalized the classical Schwarz-Pick lemma in hyperbolic geometry to the discrete Schwarz-Pick lemma for Andreev circle packings.  This paper continues to investigate the discrete Schwarz-Pick lemma for generalized circle packings (including circle, horocycle or hypercycle)  in hyperbolic background geometry. Since the discrete Schwarz-Pick lemma is to compare some geometric quantities of two generalized circle packings with different boundary values, we first show the existence and rigidity of generalized circle packings with boundary values, and then  we introduce the method of combinatorial Calabi flows to find the generalized circle packings with  boundary values.  Moreover, motivated by the method of He \cite{he},  
 we propose the maximum principle for generalized circle packings. Finally, we use the maximum principle to prove the discrete Schwarz-Pick lemma for generalized circle packings. 
	
	\end{abstract}

	\section{Introduction}

Discrete conformal structures on polyhedral surfaces serve as discrete analogues of conformal structures on smooth surfaces. Thurston \cite{thurston} used the circle packings on triangulated closed surfaces to investigate hyperbolic structures on 3-manifolds, building a deep connection between circle packings and conformal structures.

Over the past few decades, the variational principle, introduced by Colin de Verdière \cite{colin} in 1991, plays an important role in establishing the existence and rigidity of certain types of circle patterns. Further work in this area includes contributions by Bobenko and Springborn \cite{bo}, Leibon \cite{Leibon}, Rivin \cite{ri}, and Luo \cite{Luo13, Luo}.

In spherical background geometry, the functionals developed by Colin de Verdière and others are non-convex, which limits the applicability of the variational principle. This restriction led to a long period of stagnation in the use of variational methods in spherical background geometry. However, Nie \cite{nie} recently introduced a new functional with total geodesic curvatures, which is convex in spherical background geometry. Using the functional, he proved the existence and rigidity of circle patterns in spherical background geometry with prescribed total geodesic curvatures at each vertex. In \cite{GHLY}, we gave an explanation on Nie's functional as follows. It is not difficult to see that  the  version of total geodesic curvatures in Euclidean background geometry is equivalent to the case of conical angles. In \cite{nie2}, Nie rewrote Thurston-Andreev theorem that the value range of conical angles in Euclidean background geometry is the boundary of a convex polyhedron in Euclidean space $\mathbb{R}^V$.  Nie's spherical work can be seen as extending from the boundary of this convex polyhedron to its interior as the Gaussian curvature changes from 0 to 1. Moreover, there exists a connection between the convex functional involving total geodesic curvature in spherical background geometry and that in Euclidean background geometry.

 As a generalization, the second author of this paper \cite{BHS} explored the existence and rigidity of generalized hyperbolic circle packings with conical singularities on triangulated compact surfaces, which can include circles, horocycles, and hypercycles, for given total geodesic curvatures at each vertex.  It is another extension with the Gaussian curvature changing from 0 to -1.

This paper is the continuation of generalized hyperbolic circle packings in \cite{ BHS}.

Beardon and Stephenson \cite{be} studied boundary value problems for circle packings by using regular family methods. Inspired by their work, we pose the following\\

\noindent{\bf  Boundary value problem: }\; {\em Is there the generalized circle packing with prescribed boundary geodesic curvatures and internal total geodesic curvatures on $S_{g,n}$?  How to find it?}

We obtain the following theorem that solves the existence and rigidity of this problem:

\begin{theorem}\label{thm}
	Given a closed topological surface $S_{g,n}$ with a triangulation $\mathcal{T} = (V, E, F)$ as in subsection \ref{GCPS}, for any prescribed geodesic curvatures $\hat{k} \in \mathbb{R}_{>0}^{|V^{\partial}|}$ and prescribed total geodesic curvatures $\hat{T} \in \mathbb{R}^{|V^{\circ}|}_{>0}$, there exists a generalized circle packing $P$ on $S_{g,n}$ that realizes the pair $(\hat{k}, \hat{T})$ if and only if the prescribed total geodesic curvature $\hat{T} \in \mathfrak{T}$, where
	\begin{equation}
		\mathfrak{T} = \left\{ (T_1, \ldots, T_{|V^{\circ}|}) \in \mathbb{R}^{|V^{\circ}|}_{>0} \ \middle|\ \sum_{v \in I} T_v < \pi |F_I|, \ \forall I \subset V^{\circ} \right\}.
	\end{equation}
	Furthermore, the generalized circle packing $P$, if it exists, is unique up to isometry.
\end{theorem}

In 1982, Hamilton \cite{Hamilton} introduced the Ricci flow, which Perelman \cite{Perelman1, Perelman2, Perelman3} later used to prove the Poincaré conjecture. In 2003, Chow and Luo \cite{chow-Luo} introduced the combinatorial Ricci flow, which they used to find circle patterns with prescribed curvature in Euclidean and hyperbolic background geometries. Ge, in his Ph.D. thesis \cite{Ge1}, introduced the combinatorial Calabi flow to study circle packings. Significant results on combinatorial curvature flows have since emerged \cite{ke2, ke1,  Ge1, gehua,  ge3, ge4,ge8, gejsh, ge6, ge7, ge9}, including a notable contribution by Feng-Ge-Hua \cite{ke2,ke1}, which impacted 3-dimensional geometric topology through Luo’s 3-dimensional combinatorial Ricci flow \cite{Luo05}, leading to breakthroughs in Thurston’s “geometric ideal triangulation conjecture.” The methodological innovations and proof strategies introduced by Chow-Luo \cite{chow-Luo} and developed by Ge and his collaborators \cite{Ge1,Ge18,gehua, ge4,gejsh} have established a foundational framework and essential techniques in this field. It is no exaggeration to say that, in recent years, their studies have profoundly influenced in the research on circle packings and combinatorial curvature flows, shaping the development and direction of these areas.

In this paper, we will introduce the combinatorial Calabi flow to find  desired generalized  circle packings  with boundary values. 

Given geodesic curvatures $k=(k_1,\cdots,k_{|V^{\partial}|})\in\mathbb{R}_{>0}^{|V^{\partial}|}$ and prescribed total geodesic curvatures $\hat{T}=(\hat{T}_1,\cdots,\hat{T}_{|V^{\circ}|})\in\mathfrak{T}$, for geodesic curvatures $k=(k_1,\cdots,k_{|V^{\circ}|})\in\mathbb{R}_{>0}^{|V^{\circ}|}$, we define the  combinatorial Calabi flows  to find generalized circle packings with  boundary values:
\begin{definition}(combinatorial Calabi flows)
	\begin{equation}\label{f1}
		\frac{dk}{dt}=-KM'(T-\hat{T}),
	\end{equation}
	where $K=diag\{k_1,\cdots,k_{|V^{\circ}|}\}$, $M$ is the Jacobi matrix (\ref{ja}) and $M'$ is the transposition of $M$.
\end{definition}

We studied the long-time existence and convergence of the solution to combinatorial Calabi flows:
\begin{theorem}\label{t1}
	For any initial geodesic curvature $k(0)\in\mathbb{R}_{>0}^{|V^{\circ}|}$, the solution to the combinatorial Calabi flow (\ref{f1}) exists for all time $t\in[0,+\infty)$ and is unique.
\end{theorem}

\begin{theorem}\label{Calabi flow}
	For any initial geodesic curvature $k(0)\in\mathbb{R}_{>0}^{|V^{\circ}|}$, the solution to the combinatorial Calabi flow (\ref{f1}) converges to the generalized circle packing on $S_{g,n}$ that realizes $(\hat{k},\hat{T})$.
\end{theorem}

In 1985,  at the International Symposium  in Celebration of the
Proof of the Bieberbach Conjecture, Thurston conjectured that the sequence of functions converges to the Riemann mapping from some bounded simply connected domain to the unit disc, provided the packings are taken as sub-packings of scaled copies of the infinite hexagonal circle packing.  Rodin and Sullivan subsequently proved this conjecture in \cite{ro2}. Later, Rodin \cite{ro} obtained the Schwarz Lemma for the circle packings based on the hexagonal combinatorics,  which is a natural analogue for the classical Schwarz Lemma. Huang, Liu and Shen \cite{huang} generalized the Schwarz’s Lemma of Rodin for the circle packings with the general combinatorics.

Beardon and Stephenson \cite{be} discovered the discrete Schwarz-Pick Lemma for Andreev circle packings, which 
is not directly related to Rodin's work. In two-dimensional hyperbolic geometry, the Schwartz-Pick theorem provides a method for comparing the distances between different points. Beardon and Stephenson  compared hyperbolic lengths and areas between two hyperbolic surfaces induced by circle packings with different boundary values. Z. He \cite{he}  proposed a maximum principle in the hyperbolic plane to prove the rigidity property for locally finite disk patterns in the hyperbolic plane. Moreover, the second part of the maximum principle is an analogue of the discrete Schwarz-Pick Lemma. 

Building on the works of Beardon, Stephenson and He, we naturally ask the following problem:\\

\noindent{\bf Problem: }\; {\em Can we generalized the maximum principle and the discrete Schwarz-Pick lemma for generalized circle packings in hyperbolic geometry?}

We solve this problem as the following Theorems \ref{thm5}
and \ref{thm6}.

Given two sets of prescribed boundary geodesic curvatures $\hat{k}, \hat{k}^* \in \mathbb{R}_{>0}^{|V^{\partial}|}$ and interior total geodesic curvatures $\hat{T} \in \mathfrak{T}$ for the vertex sets $V^{\partial}$ and $V^{\circ}$ respectively. Theorem \ref{thm} guarantees the existence of generalized circle packings $P$ and $P^*$ on $S_{g,n}$ that realize the curvature pairs $(\hat{k}, \hat{T})$ and $(\hat{k}^*, \hat{T})$ respectively. 

Let $k(P(v))$ and $k(P^*(v))$ denote the geodesic curvatures at the vertex $v$ in the generalized circle packings $P$ and $P^*$, respectively. Then we  establish the following maximum principle for generalized hyperbolic circle packings:

\begin{theorem}\label{thm5}(Maximum Principle in Hyperbolic Background Geometry)
		Let $P$ and $P^*$ be as above. Then:\\
		(a) The maximum of $k(P^*(v))/k(P(v))$, if $>1$, is never attained at an interior vertex; and\\
		(b) In particular, if the inequality $k(P^*(v))\le k(P(v))$ is true for each boundary vertex, then it holds for all vertices.
	\end{theorem}

Using the maximum principle in hyperbolic background geometry, we establish and prove the following discrete Schwartz-Pick lemma for generalized hyperbolic circle packings:

\begin{theorem}\label{thm6}(Discrete Schwartz-Pick Lemma)
		Let $\operatorname{Area}(\Omega_f),\operatorname{Area}(\Omega^*_f),l_v^f, l_v^{f*}, d_{ij}$ and $d_{ij}^*$ be as in section \ref{sd}. If the inequality $k(P^*(v))\le k(P(v))$ is true for each boundary vertex, then:\\
		(a) $\operatorname{Area}(\Omega^*_f)\ge\operatorname{Area}(\Omega_f)$ for every face $f$; \\
		(b) $l_v^{f*}\ge l_v^f$ for every face $f$ and every vertex $v$ of $f$;\\
		(c) $d_{ij}^*\ge d_{ij}$ for any two tangent points $i,j\in\hat{V}^*$ and corresponding tangent points $i^*,j^*\in\hat{V}^*$.
	\end{theorem}	
	
\noindent\textbf{Organization} This paper is organized as follows. In section \ref{22}, we introduce some preliminaries of generalized circle packings. In section \ref{s1}, we construct some potential functions for local generalized circle packings and prove the strict convexity of these functions. In section \ref{s2}, we construct the potential function for generalized circle packings and prove the strict convexity of the function. In section \ref{s3}, we prove the existence and rigidity of generalized circle packings on $S_{g,n}$ (cf. Theorem \ref{thm}). In section \ref{s4}, we define the combinatorial Calabi flows for prescribed total geodesic curvatures and study the convergence of solutions to the flows (cf. Theorems \ref{t1} and \ref{Calabi flow}). In section \ref{sd}, we prove the maximum principle in hyperbolic background geometry and the discrete Schwartz-Pick lemma for generalized circle packings in hyperbolic background geometry(cf. Theorems \ref{thm5} and \ref{thm6}).

	\section{Preliminary}\label{22}	
	\subsection{Generalized circle packings on triangles}\label{GCPT}
	In Poincar\'{e} disk model $\mathbb{H}^2$, a \textbf{generalized circle} can be a circle, a horocycle, or a hypercycle. It is worthy to mention that, the center of a horocycle is its unique limiting ideal point, and the center of a hypercycle is the geodesic that connects its two limiting ideal points, as shown in Figure \ref{6}.
	
	\begin{figure}[ht] 
		\centering 
		\includegraphics[width=0.3\textwidth]{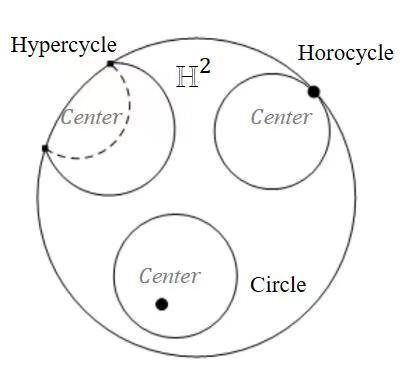} 
		\caption{\small The generalized circles on $\mathbb{H}^2$.}
		\label{6} 
	\end{figure}
	
	For a generalized circle, the radius, which is the distance between a point on it with the center, is constant. It is known that the shape of a generalized circle can be determined by its constant geodesic curvature \( k \). The relation between the radius $r$ and the constant geodesic curvature \( k \) is 
	\begin{equation}
		r(k)= \begin{cases}\operatorname{arctanh} k & \text { if } k<1 \\ +\infty & \text { if } k=1 \\ \operatorname{arccoth} k & \text { if } k>1,\end{cases}
	\end{equation}
	where the generalized circle is a hypercycle if $k<1$, a horocycle if $k=1$ and a circle if $k>1$.
	
	The inner angle $\theta$ of the arc $C$ for circle and horocycle is defined as the interior angle of two shortest geodesics, which connects the end points of arc $C$ and center. For hypercycle, the inner angle $\theta$ is defined as the hyperbolic length of shortest distance projection of $C$ onto the geodesic axis $\gamma$, as shown in Figure \ref{7}.
	
	\begin{figure}[ht]
		\centering
		\includegraphics[scale=0.3]{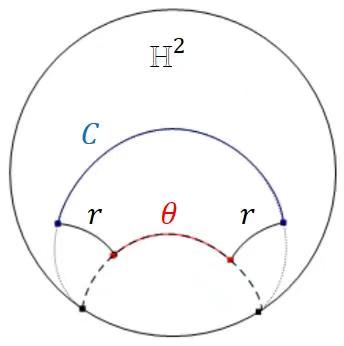}
		\captionof{figure}{\small The inner angle $\theta$ of the arc $C$ for hypercycle.}
		\label{7}
	\end{figure}
	The Table \ref{relation} shows the relations among the radius $r$, geodesic curvature $k$, inner angle $\theta$ and the corresponding arc length $l$. For the case of horocycle, the arc length is shown in \cite[Lemma 2.5]{BHS}. 
	\begin{table}[ht]
		\centering
		\caption{\small The relation of radii, absolute values of geodesic curvatures and arc lengths}
		\begin{tabular}{|l|l|l|l|}
			\hline & Radius & Absolute value of geodesic curvature & Arc length \\
			\hline Circle & $0<r<\infty$ & $k=\operatorname{coth} r$ & $l=\theta \sinh r$ \\
			\hline Horocycle & $r=\infty$ & $k=1$ &- \\
			\hline Hypercycle & $0<r<\infty$ & $k=\tanh r$ & $l=\theta \cosh r$ \\
			\hline
		\end{tabular}
		\label{relation}
	\end{table}
	
	\begin{figure}
		\centering
		\includegraphics[scale=0.25]{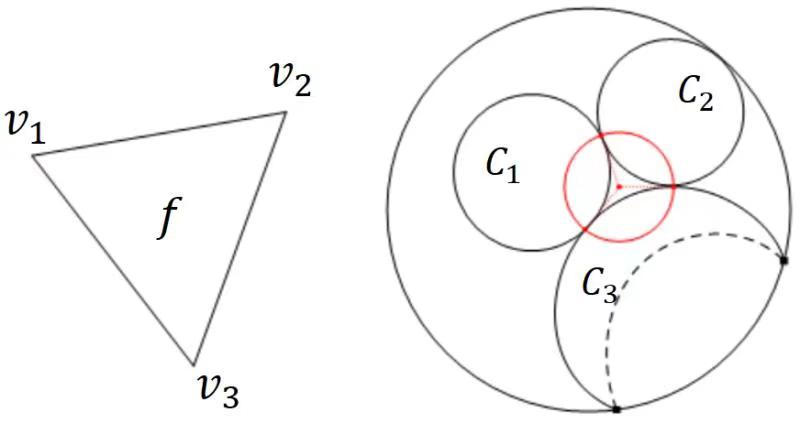}
		\captionof{figure}{\small The generalized circle packing of the triangle $f$.}
		\label{5}
	\end{figure}
 
	We need the following result for generalized circle packing in hyperbolic background geometry.
	
	\begin{lemma}(\cite[Lemma 2.4]{BHS})\label{l1}
		Given a triangle $f$ with vertices $v_1,v_2,v_3$. For any $k=(k_1,k_2,k_3)\in \mathbb{R}_{>0}^{3}$, then there exists a geometric pattern $P(f)$ formed by three generalized circles $C_1,C_2,C_3$ on $\mathbb{H}^2$ which satisfies the following conditions:
		\begin{enumerate}
			\item The generalized circles $C_1$, $C_2$ and $C_3$ are mutually externally tangent.
			\item There exists a hyperbolic circle $C_f$ which perpendiculars to each $C_i$ and contains all tangent points of  $C_1,C_2,C_3.$ We call $C_f$ the dual circle of $f$.
			\item The geodesic curvature of generalized circle $C_i$ is $k_i, i=1,2,3$. 
		\end{enumerate} 
		The geometric pattern $P(f)$ is also called a \textbf{generalized circle packing} with geodesic curvatures $k=(k_1,k_2,k_3)$ of the triangle $f$ as shown in Figure \ref{5}. 
	\end{lemma}

	For convenience, we also denote the center of $C_i$ by $v_i$ in this paper.
	
	By Lemma \ref{l1}, the generalized circle packing $P(f)$ is determined by the geodesic curvatures $k_1,k_2,k_3$ of respective vertices $v_1,v_2,v_3$. 
	
	As shown in in Figure \ref{4}, if we add the shortest geodesics segment $v_{i}v_{j}$ that connect $v_i$ and $v_j$, we can obtain a hyperbolic polygon $\hat{f}$. Let $l_{ij}$ be the length of $v_{i}v_{j}$. It is obvious that $l_{ij}= r(k_i)+r(k_j)$.

	\begin{figure}[htbp]
		\centering
		\includegraphics[scale=0.3]{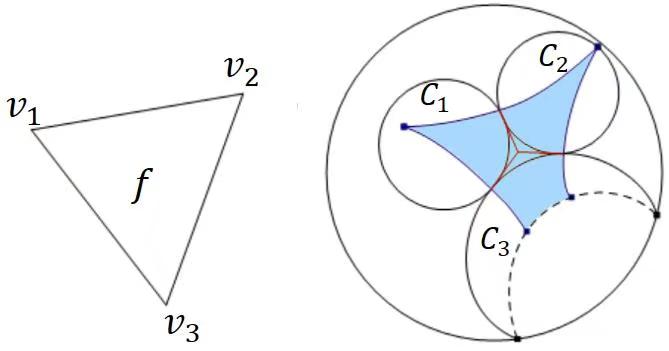}
		\captionof{figure}{\small The triangle $f$ and hyperbolic triangle $\hat{f}$.}
		\label{4}
	\end{figure}
	
	\begin{remark}
		Since there are three types of generalized circle, so there are ten types of hyperbolic polygon that $\hat{f}$ could be, as shown in Figure \ref{3}. Since we regard the center of a hypercycle as a point, we would also say $\hat{f}$ is a \textbf{hyperbolic triangle}.
		
	\end{remark}

	In this paper, the hyperbolic triangle $\hat{f}$ is called  hyperbolic triangle of generalized circle packing $P(f)$ with geodesic curvatures $(k_1,k_2,k_3)$ on the triangle $f$.

	\subsection{Generalized circle packing on surfaces}\label{GCPS}
	Let $S_{g,n}$ be a surface with boundary, which is obtained by removing $n\ge 1$ topological open disks from an oriented compact surface with genus $g\geq0$. Given a triangulation $\mathcal{T}=(V,E,F)$ of $S_{g,n}$, where $V$, $E$ and $F$ denote the set of vertices, edges and faces, respectively. In this paper, there exists at least two boundary vertices in $V$ and for each $f\in F$, at most two vertices of $f$ are boundary vertices, as shown in Figure \ref{1}.
	
		\begin{figure}[htbp]
		\centering
		\includegraphics[scale=0.25]{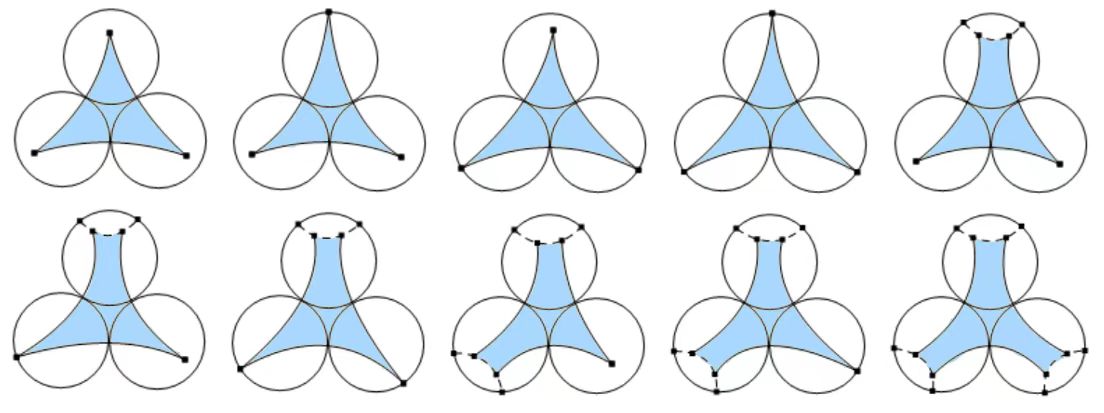}
		\captionof{figure}{\small The hyperbolic triangles of $f$.}
		\label{3}
	\end{figure}
	Given a topological surface $S_{g,n}$ with a triangulation $\mathcal{T}=(V,E,F)$, for any geodesic curvature $k\in \mathbb{R}^{|V|}_{>0}$ on $V$, we can construct a new surface $\hat{S}_{g,n}$ as follows.
	\begin{enumerate}
		\item For each face $f\in F$, by section \ref{s1}, we can construct a 
		hyperbolic triangle $\hat{f}$ of generalized circle packing $P(f)$ with geodesic curvatures $(k_{v_1},k_{v_2},k_{v_3})$ on the face $f$.
		\item Gluing all the hyperbolic triangles together along their edges, then we obtain a surface $\hat{S}_{g,n}$ with a hyperbolic conical metric, the metric is called a \textbf{generalized circle packing metric}.
	\end{enumerate}

	All the generalized circle packings form a new geometric pattern $P$, the geometric pattern $P$ is also called a \textbf{generalized circle packing} on topological surface $S_{g,n}$. 
	
	\begin{figure}[htbp]
		\centering
		\includegraphics[scale=0.8]{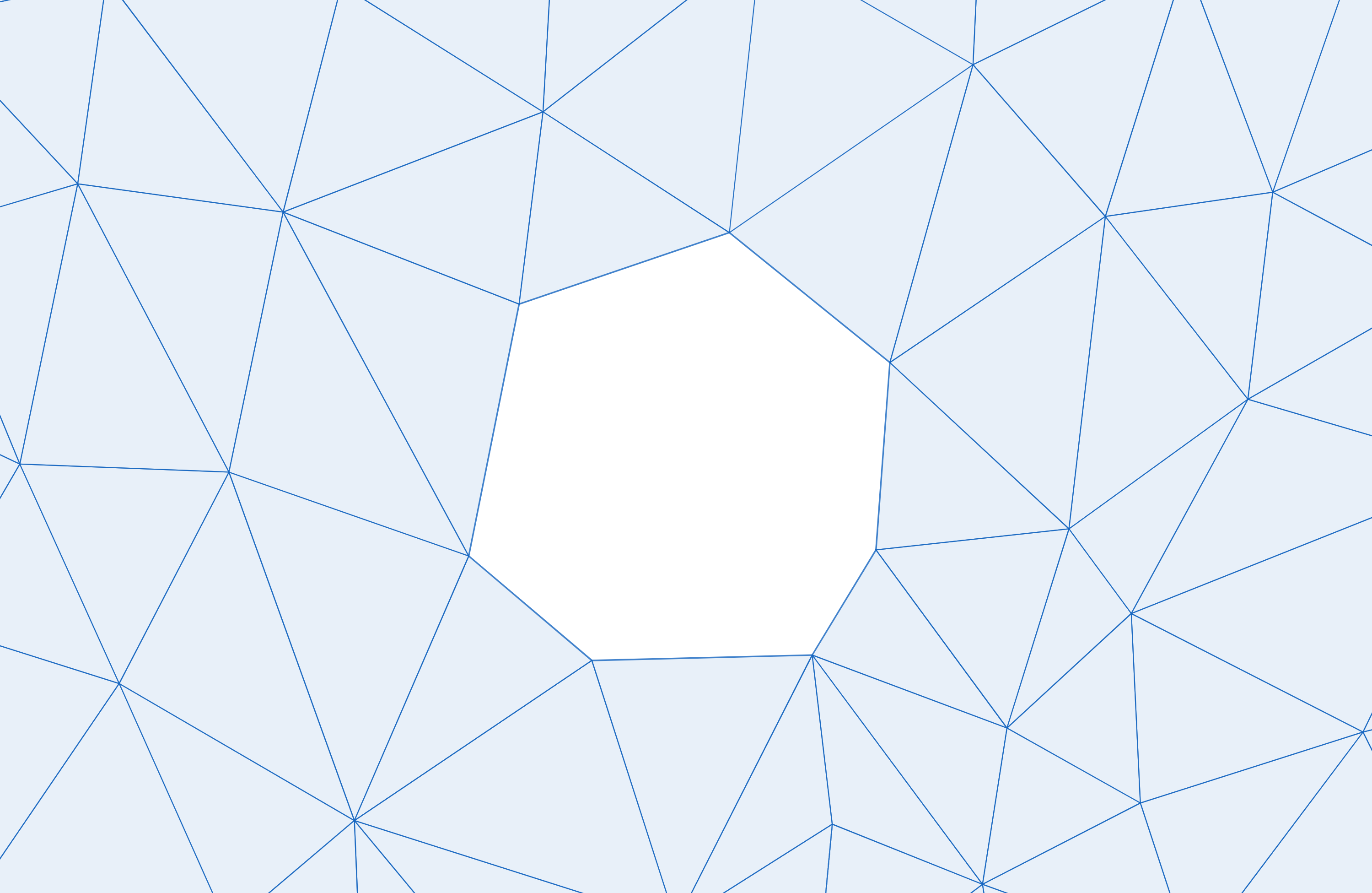}
		\captionof{figure}{\small The triangulation of $S_{g,n}$.}
		\label{1}
	\end{figure} 
	
	Let $V_1,V_2$ and $V_3$ be the sets of vertices defined as follows.
	\begin{enumerate}
		\item $v\in V_1$ iff $0<k_{v}<1$.
		\item $v\in V_2$ iff $k_{v}=1$.
		\item $v\in V_3$ iff $k_{v}>1$.
	\end{enumerate}
	Topologically, $\hat{S}_{g,n}$ can be obtained by removing a disk or a half disk (a point) for each $v\in V_1$($v\in V_2$) from $S_{g,n}$ respectively.

	\section{Potential functions for local generalized circle packing}\label{s1}

	By $l_{i}^f$ we denote the length of the sub-arc $C_{i}^f$ of $C_i$ contained in dual circle $C_f$, then we have the following lemma.
	
	\begin{lemma}\label{l2}
		If $k_1,k_2$ and $k_3$ are not all constants, then the differential 1-form
		\[
		\eta=\sum_{i=1}^3l_{i}^f\mathrm{d}k_i
		\]
		is closed.
	\end{lemma}
	
	\begin{proof}
		
		If two of $k_1, k_2$ and $k_3$ are constants, we can assume that $k_1$ and $k_2$ are constants, then $l_{3}^f$ is a smooth function with respect to $k_3$ and $\eta=l_{3}^f\mathrm{d}k_3$. Hence we obtain that $\mathrm{d}\eta=0$, the 1-form $\eta$ is closed. 
		
		If one of $k_1, k_2$ and $k_3$ is a constant, we can assume that $k_3$ is a constant, then $\eta=l_{1}^f\mathrm{d}k_1+l_{2}^f\mathrm{d}k_2$. By $l_f$ we denote the length of dual circle $C_f$ and by $\tilde{l}_{i}^{f}$ we denote the length of the sub-arc $\tilde{C}_{i}^{f}$ of $C_f$ contained in $C_i$ as shown in Figure \ref{l2}. Then we know that $l_f=\tilde{l}_{1}^{f}+\tilde{l}_{2}^{f}+\tilde{l}_{3}^{f}$ and $l_f$ is a smooth function with respect to $k_f$, where $k_f$ is the geodesic curvature of dual circle $C_f$.
		\begin{figure}[htbp]
			\centering
			\includegraphics[scale=0.80]{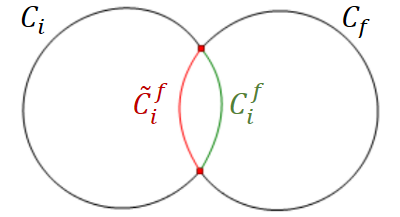}
			\captionof{figure}{\small The sub-arc $C_{i}^{f}$ and $\tilde{C}_{i}^{f}$.}
			\label{2}
		\end{figure} 
		
		Then we have
		\begin{align*}
			\eta+l_f\mathrm{d}k_f&=l_{1}^{f}\mathrm{d}k_1+l_{2}^{f}\mathrm{d}k_2+l_f\mathrm{d}k_f\\
			&=(l_{1}^{f}\mathrm{d}k_1+\tilde{l}_{1}^{f}\mathrm{d}k_f)+(l_{2}^{f}\mathrm{d}k_2+\tilde{l}_{2}^{f}\mathrm{d}k_f)+\tilde{l}_{3}^{f}\mathrm{d}k_f.
		\end{align*}
		By Lemma 2.7 in \cite{BHS}, we know that the 1-form $l_{i}^f\mathrm{d}k_i+\tilde{l}_{i}^{f}\mathrm{d}k_f$ is closed, $i=1,2$. Since $k_3$ is a constant, then $\tilde{l}_{3}^{f}$ is a smooth function with respect to $k_f$. Hence we obtain that $\mathrm{d}\eta=0$, the 1-form $\eta$ is closed.  
		
		If $k_1,k_2$ and $k_3$ are not constants, by Lemma 2.8 in \cite{BHS}, the 1-form $\eta$ is closed.
	\end{proof}
	
	By $T_i^f=l_i^f\cdot k_i$ we denote the total geodesic curvature of sub-arc $C_i^f$ of $C_i$ contained in dual circle $C_f$, $i=1,2,3$. Set $s_i=\ln k_i$, $i=1,2,3$ and by Lemma \ref{l2}, we know that if $k_1,k_2$ and $k_3$ are not all constants, then the differential 1-form
	\[\omega_{f}=\sum_{i=1}^3 T_i^f \mathrm{d}s_{i},\]
	is closed. If $k_1,k_2$ and $k_3$ are not constants, then we can define a potential function on $\mathbb{R}^3$ as follow: 
	\[\mathcal{E}_f(s_1,s_2,s_3)=\int^{(s_1,s_2,s_3)}_{0}\omega_{f}.\]
	If only $k_1$ is a constant, then we can define a potential function on $\mathbb{R}^2$ as follow:
	\[\mathcal{E}_f^1(s_2,s_3)=\int^{(s_2,s_3)}_{0}\omega_{f}.\]
	If only $k_2$ is a constant or only $k_3$ is a constant, then we can similarly define the potential functions $\mathcal{E}_f^2(s_1,s_3)$ and $\mathcal{E}_f^3(s_1,s_2)$.
	
	If only $k_1$ is not a constant, then we can define a potential function on $\mathbb{R}$ as follow:
	\[\mathcal{E}_f^1(s_1)=\int^{s_1}_{0}\omega_{f}.\]
	If only $k_2$ is not a constant or only $k_3$ is not a constant, then we can similarly define the potential functions $\mathcal{E}_f^2(s_2)$ and $\mathcal{E}_f^3(s_3)$. These potential functions are well-defined. Moreover, we have the following proposition.
	
	\begin{proposition}\label{pro1}
		These potential functions are strictly convex. 
	\end{proposition}
	
	\begin{proof}
		If $k_1,k_2$ and $k_3$ are not constants, then we have 
		$$
		\nabla\mathcal{E}_f=(T_1^f,T_2^f,T_3^f)^{\prime},~~\text{Hess}~\mathcal{E}_f=\left(\frac{\partial T_i^f}{\partial s_j}\right)_{3\times 3}. $$
		By Lemma 2.11 in \cite{BHS}, we have that $\text{Hess}~\mathcal{E}_f$ is positive definite and $\mathcal{E}_f(s_1,s_2,s_3)$ is strictly convex.
		
		If only $k_3$ is a constant, for the potential function $\mathcal{E}_f^3(s_1,s_2)$, we have
		$$
		\nabla\mathcal{E}_f^3=(T_1^f,T_2^f)^{\prime},~~\text{Hess}~\mathcal{E}_f^3=\left(\frac{\partial T_i^f}{\partial s_j}\right)_{2\times 2}, $$
		where $T_i^f=T_i^f(s_1,s_2), i=1,2$. Let $\Omega_f$ be the region enclosed by sub-arcs $C_i^f, i=1,2,3$ as shown in Figure \ref{11}. 
		\begin{figure}[htbp]
			\centering
			\includegraphics[scale=0.30]{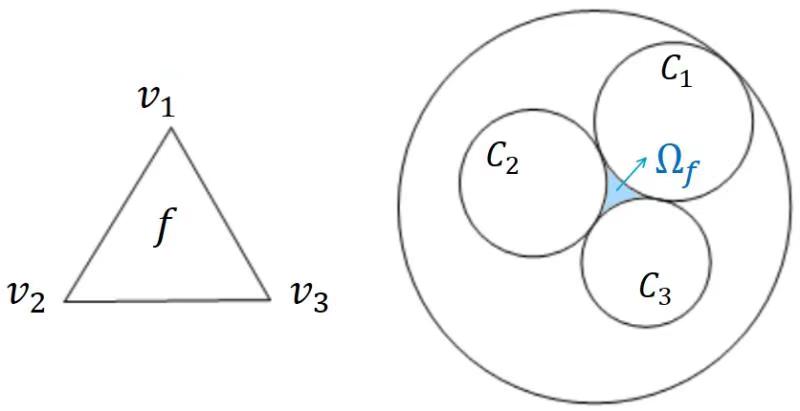}
			\captionof{figure}{\small An example of face $f$ and region $\Omega_{f}$.}
			\label{11}
		\end{figure}

		By Gauss-Bonnet theorem, we have 
		\begin{equation}\label{g1}
			\area(\Omega_f)=\pi-\sum_{i=1}^{3}T_i^f.
		\end{equation}
		By Lemma 2.11 in \cite{BHS}, we have  
		$$
		\frac{\partial T_i^f}{\partial s_i}>0,~~i=1,2,
		$$
		and 
		$$
		\frac{\partial T_i^f}{\partial s_j}<0,~~1\le i\ne j\le 2.
		$$
		Then we obtain 
		\begin{equation}\label{g2}
			\left\vert\frac{\partial T_1^f}{\partial s_1}\right\vert-\left\vert\pp{T_2^f}{s_1}\right\vert=\pp{T_1^f}{s_1}+\pp{T_2^f}{s_1}=\pp{(T_1^f+T_2^f)}{s_1}.
		\end{equation}
		By (\ref{g1}), we have 
		$$
		T_1^f+T_2^f=\pi-T_3^f-\area(\Omega_f),$$
		and
		\begin{equation}\label{g3}
			\pp{(T_1^f+T_2^f)}{s_1}=-\pp{T_3^f}{s_1}-\pp{\area(\Omega_f)}{s_1}.
		\end{equation}
		By Lemma 2.11 in \cite{BHS}, we have 
		\begin{equation}\label{g4}
			\pp{T_3^f}{s_1}<0,~~\pp{\area(\Omega_f)}{s_1}<0.
		\end{equation}
		By (\ref{g2}), (\ref{g3}) and (\ref{g4}), we obtain
		$$
		\left\vert\frac{\partial T_1^f}{\partial s_1}\right\vert-\left\vert\pp{T_2^f}{s_1}\right\vert=\pp{(T_1^f+T_2^f)}{s_1}>0.
		$$
		Using above similar argument, we have
		$$
		\left\vert\frac{\partial T_2^f}{\partial s_2}\right\vert-\left\vert\pp{T_1^f}{s_2}\right\vert=\pp{(T_1^f+T_2^f)}{s_2}>0.
		$$
		Then we know that $\text{Hess}~\mathcal{E}_f^3$ is a symmetric strictly diagonally dominant matrix with positive diagonal entries, which yields that $\text{Hess}~\mathcal{E}_f^3$ is positive definite. Hence the potential function $\mathcal{E}_f^3(s_1,s_2)$ is strictly convex. We can also similarly prove that potential functions $\mathcal{E}_f^1(s_2,s_3)$ and $\mathcal{E}_f^2(s_1,s_3)$ are strictly convex. 
		
		If only $k_1$ is not a constant, for the potential function $\mathcal{E}_f^1(s_1)$, we have $\frac{\mathrm{d}\mathcal{E}_f^1}{\mathrm{d}s_1}=T_1^f$, then by Lemma 2.11 in \cite{BHS}, we have $$\frac{\mathrm{d}^2\mathcal{E}_f^1}{\mathrm{d}s_1^2}=\frac{\mathrm{d}T_1^f}{\mathrm{d}s_1}>0.
		$$
		Hence the potential function $\mathcal{E}_f^1(s_1)$ is strictly convex. We can also similarly prove that potential functions $\mathcal{E}_f^2(s_2)$ and $\mathcal{E}_f^3(s_3)$ are strictly convex. 
	\end{proof}
	
	\section{The potential function for generalized circle packing}\label{s2}
	
	Given a surface $S_{g,n}$ with boundary and triangulation $\mathcal{T}=(V,E,F)$ as in section \ref{s1}. By $V^{\circ}$ we denote the set of interior vertices in $V$ and by $V^{\partial}$ we denote the set of boundary vertices in $V$, then we have $V=V^{\circ}\cup V^{\partial}$. 
	
	All the vertices in $V^{\circ}$ are ordered one by one, marked by $1,\cdots,|V^{\circ}|$ and all the vertices in $V^{\partial}$ are ordered one by one, marked by $1,\cdots,|V^{\partial}|$, where $|V^{\circ}|$ and $|V^{\partial}|$ are the number of vertices in $V^{\circ}$ and the number of vertices in $V^{\partial}$, respectively.
	
	Given geodesic curvatures \(\hat{k} = (k_1, \ldots, k_{|V^{\partial}|}) \in \mathbb{R}^{|V^{\partial}|}_{>0}\) on \(V^{\partial}\), we can define the total geodesic curvature \(T_i\) at the vertex \(i \in V^{\circ}\) for any \(k = (k_1, \ldots, k_{|V^{\circ}|}) \in \mathbb{R}^{|V^{\circ}|}_{>0}\) on \(V^{\circ}\). Specifically, we have:
	
	\[
	T_i := \sum_{f \in F_{\{i\}}} T^f_{i}
	\]
	where \(F_{\{i\}}\) is the set of all faces containing \(i\) as one of its vertices.

	By $F_i$ we denote the set $F_i= \{f \in F~|~|V_f\cap V^{\partial}|=i\}$, then we know that $F_1\cup F_2\ne \emptyset$ and $F=F_0\cup F_1\cup F_2$. For any face $f\in F$, we can suppose that $V_f=\{v_1(f),v_2(f),v_3(f)\}$. For any face $f\in F_1$, we can suppose that $V_f\cap V^{\partial}=\{v_1(f)\}$ and for any face $f\in F_2$, we can suppose that $V_f\cap V^{\partial}=\{v_2(f),v_3(f)\}$. 
	
	Using change of variables $s_i=\ln k_i$, $i=1,\cdots,|V^{\circ}|$ and $\hat{s}_i=\ln \hat{k}_i$, $i=1,\cdots,|V^{\partial}|$, we consider the new variable $s=(s_1,\cdots,s_{|V^{\circ}|})\in \mathbb{R}^{|V^{\circ}|}$, then we can define a potential function 
	$$
	\mathcal{E}(s):=\sum_{f\in F_0}\mathcal{E}_f(s_{v_1(f)},s_{v_2(f)},s_{v_3(f)})+\sum_{f\in F_1}\mathcal{E}_f^{v_1(f)}(s_{v_2(f)},s_{v_3(f)})+\sum_{f\in F_2}\mathcal{E}_f^{v_1(f)}(s_{v_1(f)}).
	$$
	Moreover, for $1\le i\le |V^{\circ}|$, we have 
	\begin{align*}
		\pp{\mathcal{E}(s)}{s_i}&=\sum_{f\in F_0\cap F_{\{i\}}}\pp{\mathcal{E}_f(s_{v_1(f)},s_{v_2(f)},s_{v_3(f)})}{s_i}+\sum_{f\in F_1\cap F_{\{i\}}}\pp{\mathcal{E}_f^{v_1(f)}(s_{v_2(f)},s_{v_3(f)})}{s_i}+\sum_{f\in F_2\cap F_{\{i\}}}\pp{\mathcal{E}_f^{v_1(f)}(s_{v_1(f)})}{s_i}\\
		&=\sum_{f\in F_0\cap F_{\{i\}}}T_i^f(s_{v_1(f)},s_{v_2(f)},s_{v_3(f)})+\sum_{f\in F_1\cap F_{\{i\}}}T_i^f(s_{v_2(f)},s_{v_3(f)})+\sum_{f\in F_2\cap F_{\{i\}}}T_i^f(s_{v_1(f)})\\
		&=\sum_{f\in F_{\{i\}}}T_i^f\\
		&=T_i.
	\end{align*}
	Then we have $\nabla\mathcal{E}=(T_1,\cdots,T_{|V^{\circ}|})$ and the Hessian of $\mathcal{E}$ is a Jacobi matrix, i.e.
\begin{equation}\label{ja}	
 \text{Hess}~\mathcal{E}=M=\begin{pmatrix}
		\frac{\partial T_1}{\partial s_1}&\cdots&\frac{\partial T_1}{\partial s_{|V^{\circ}|}}\\
		\vdots&\ddots&\vdots\\
		\frac{\partial T_{|V^{\circ}|}}{\partial s_{1}}&\cdots&\frac{\partial T_{|V^{\circ}|}}{\partial s_{|V^{\circ}|}}\\
	\end{pmatrix}.
 \end{equation}
	
	\begin{proposition}\label{pro3}
		The Jacobi matrix $M$ is positive definite.
	\end{proposition}
	
	\begin{proof}
		For any $1\le i\le |V^{\circ}|$, we have
		\begin{equation}\label{e1}
			\pp{T_i}{s_i}=\sum_{f\in F_0\cap F_{\{i\}}}\pp{T_i^f(s_{v_1(f)},s_{v_2(f)},s_{v_3(f)})}{s_i}+\sum_{f\in F_1\cap F_{\{i\}}}\pp{T_i^f(s_{v_2(f)},s_{v_3(f)})}{s_i}+\sum_{f\in F_2\cap F_{\{i\}}}\pp{T_i^f(s_{v_1(f)})}{s_i}.
		\end{equation}
		By Lemma 2.11 in \cite{BHS} and Proposition \ref{pro1}, we have 
		$$
		\pp{T_i^f(s_{v_1(f)},s_{v_2(f)},s_{v_3(f)})}{s_i}>0~~ \pp{T_i^f(s_{v_2(f)},s_{v_3(f)})}{s_i}>0,~~ \pp{T_i^f(s_{v_1(f)})}{s_i}>0. 
		$$
		Since there exists at least one of three sums in (\ref{e1}), then we obtain that $\frac{\partial T_i}{\partial s_i}>0$ for $1\le i\le |V^{\circ}|$. 
		
		For $1\le i\ne j\le |V^{\circ}|$, we have 
		$$
		\frac{\partial T_i}{\partial s_j}=\sum_{f\in F_0\cap F_{\{i\}}}\pp{T_i^f(s_{v_1(f)},s_{v_2(f)},s_{v_3(f)})}{s_j}+\sum_{f\in F_1\cap F_{\{i\}}}\pp{T_i^f(s_{v_2(f)},s_{v_3(f)})}{s_j}.
		$$
		For any $f\in F_0\cap F_{\{i\}}$, we know that $i\in V_f$. If $j\in V_f$, by Lemma 2.11 in \cite{BHS}, then we have $$\pp{T_i^f(s_{v_1(f)},s_{v_2(f)},s_{v_3(f)})}{s_j}<0.$$ 
		If $j\notin V_f$, then we have 
		$$\pp{T_i^f(s_{v_1(f)},s_{v_2(f)},s_{v_3(f)})}{s_j}=0.$$
		For any $f\in F_1\cap F_{\{i\}}$, we know that $i\in V_f$. If $j\in V_f$, by Proposition \ref{pro1}, then we have
		$$
		\pp{T_i^f(s_{v_2(f)},s_{v_3(f)})}{s_j}<0.
		$$
		If $j\notin V_f$, then we have
		$$
		\pp{T_i^f(s_{v_2(f)},s_{v_3(f)})}{s_j}=0.
		$$
		Then we obtain $\frac{\partial T_i}{\partial s_j}\le 0$ for $1\le i\ne j\le |V^{\circ}|$.
		
		Hence for $1\le i\le |V^{\circ}|$, we have
		\begin{align*}
			\left|\frac{\partial T_i}{\partial s_i}\right|-\sum_{1\le j \ne i\le |V^{\circ}|}\left|\frac{\partial T_j}{\partial s_i}\right|&=\frac{\partial T_i}{\partial s_i}+\sum_{1\le j \ne i\le |V^{\circ}|}\frac{\partial T_j}{\partial s_i}=\sum_{v\in V^{\circ}}\pp{T_v}{s_i}=\pp{(\sum_{v\in V^{\circ}}T_v)}{s_i}\\
			&=\pp{(\sum_{v\in V^{\circ}}\sum_{f\in F_{\{v\}}}T_v^f)}{s_i}=\sum_{v\in V^{\circ}}\sum_{f\in F_{\{v\}}}\pp{T_v^f}{s_i}\\
			&=\sum_{f\in F_{\{i\}}}\sum_{v\in V^{\circ}\cap V_f}\pp{T_v^f}{s_i}.
		\end{align*}
		For any $f\in F_{\{i\}}$, we know that $i\in V^{\circ}\cap V_f$. Then we can suppose that $V_f=\{i,j,k\}$.\\ 
		If $V^{\circ}\cap V_f=\{i\}$, by Proposition \ref{pro1}, we have
		$$
		\sum_{v\in V^{\circ}\cap V_f}\pp{T_v^f}{s_i}=\pp{T_i^f}{s_i}>0.
		$$
		If $V^{\circ}\cap V_f=\{i,j\}$, then we have 
		$$
		\sum_{v\in V^{\circ}\cap V_f}\pp{T_v^f}{s_i}=\pp{T_i^f}{s_i}+\pp{T_j^f}{s_i}.
		$$
		By Gauss-Bonnet theorem, we have
		\begin{equation}\label{g2}
			\area(\Omega_f)=\pi-T_i^f-T_j^f-T_k^f,
		\end{equation}
		then we obtain 
		$$
		\pp{T_i^f}{s_i}+\pp{T_j^f}{s_i}=-\pp{T_k^f}{s_i}-\pp{\area(\Omega_f)}{s_i}.
		$$
		By Lemma 2.11 in \cite{BHS}, we have
		$$
		\pp{T_k^f}{s_i}<0,~~\pp{\area(\Omega_f)}{s_i}<0,
		$$
		then we obtain
		$$
		\sum_{v\in V^{\circ}\cap V_f}\pp{T_v^f}{s_i}=\pp{T_i^f}{s_i}+\pp{T_j^f}{s_i}>0.
		$$
		If $V^{\circ}\cap V_f=\{i,j,k\}$, then by (\ref{g2}) and Lemma 2.11 in \cite{BHS}, we have 
		$$
		\sum_{v\in V^{\circ}\cap V_f}\pp{T_v^f}{s_i}=\pp{T_i^f}{s_i}+\pp{T_j^f}{s_i}+\pp{T_k^f}{s_i}=-\pp{\area(\Omega_f)}{s_i}>0.
		$$
		Then we obtain
		$$
		\left|\frac{\partial T_i}{\partial s_i}\right|-\sum_{1\le j \ne i\le |V^{\circ}|}\left|\frac{\partial T_j}{\partial s_i}\right|=\sum_{f\in F_{\{i\}}}\sum_{v\in V^{\circ}\cap V_f}\pp{T_v^f}{s_i}>0.
		$$
		Hence $M$ is a strictly diagonally dominant matrix with positive diagonal entries, i.e. $M$ is positive definite.
	\end{proof}
	
	\begin{corollary}\label{coro1}
		The potential function $\mathcal{E}$ is strictly convex on $\mathbb{R}^{|V^{\circ}|}$. 
	\end{corollary}
	
	\begin{proof}
		By Proposition \ref{pro3}, $\text{Hess}~\mathcal{E}$ is positive definite, then $\mathcal{E}$ is strictly convex on $\mathbb{R}^{|V^{\circ}|}$. 
	\end{proof}
	
	\section{Existence and rigidity of generalized circle packing on $S_{g,n}$ }\label{s3}

	We need the following result in convergence of total geodesic curvatures of arcs.  
	
	\begin{lemma}(\cite{BHS}, Lemma 2.12)\label{2.12}
		Given a face $f\in F$ with a vertex set $V_{f}$ consisting of $v_1,v_2,v_3$. By $\mathcal{P}=\{C_i\}_{i=1}^3$ we denote the generalized circle packing with geodesic curvatures $k=(k_1,k_2,k_3)\in\mathbb{R}^3_{>0}$ of the face $f$ and by $T_i^f$ we denote the total geodesic curvature of sub-arc $C_i^f$ of $C_i$ contained in dual circle $C_f$, $i=1,2,3$. Let $0 \leq a,b<+\infty$ and $\{r, s, t\}=\{1,2,3\}$, then the following statements hold:
		\begin{equation}\label{3.3}
			\lim _{k_r \rightarrow 0} T_r^f=0,
		\end{equation}
		\begin{equation}\label{3.4}
			\lim _{\left(k_r, k_s, k_t\right) \rightarrow(+\infty, a, b)} T_r^f=\pi,
		\end{equation}
		\begin{equation}\label{3.5}
			\lim _{\left(k_r, k_s, k_t\right) \rightarrow(+\infty,+\infty, a)} T_r^f+T_s^f=\pi,
		\end{equation}
		\begin{equation}\label{3.6}
			\lim _{\left(k_r, k_s, k_t\right) \rightarrow(+\infty,+\infty,+\infty)} T_r^f+T_s^f+T_t^f=\pi.
		\end{equation}
	\end{lemma}
	
	For a subset $I\subset V^{\circ}$, by $F_I$ we denote the set
	$$
	F_I=\{f\in F~|~\exists v\in I, s.t. ~v\in V_f\}.
	$$
	Then we have the following theorem, i.e.
	
	\begin{theorem}\label{thm1}
		$\nabla\mathcal{E}$ is a homeomorphism from $\mathbb{R}^{|V^{\circ}|}$ to $\mathfrak{T}$, where 
		$$
		\mathfrak{T}=\{(T_1,\cdots,T_{|V^{\circ}|})\in \mathbb{R}^{|V^{\circ}|}_{>0}~|~\sum_{v\in I}T_v<\pi|F_I|, \forall I\subset V^{\circ}\}.
		$$
	\end{theorem}
	
	\begin{proof}
		By the definition of $\mathcal{E}$, we have 
		$$
		\begin{aligned}
			\nabla\mathcal{E}: \mathbb{R}^{|V^{\circ}|} & \longrightarrow \mathbb{R}^{|V^{\circ}|}_{>0}  \\
			s=(s_1,\cdots,s_{|V^{\circ}|}) & \mapsto T=(T_1,\cdots,T_{|V^{\circ}|}).
		\end{aligned}
		$$
		For any subset $I\subset V^{\circ}$, we obtain
		$$
		\sum_{v\in I}T_v=\sum_{v\in I}\sum_{f\in F_{\{v\}}}T_v^f=\sum_{f\in F_I}\sum_{v\in V_f\cap I}T_v^f.
		$$
		By Gauss-Bonnet theorem, we have
		$$
		\text{Area}(\Omega_f)=\pi-\sum_{v\in V_f}T_v^f>0,
		$$
		i.e. $\sum_{v\in V_f}T_v^f<\pi$. Hence we obtain
		$$
		\sum_{v\in I}T_v=\sum_{f\in F_I}\sum_{v\in V_f\cap I}T_v^f\le \sum_{f\in F_I}\sum_{v\in V_f}T_v^f<\sum_{f\in F_I}\pi=\pi|F_I|.
		$$
		Hence $T=(T_1,\cdots,T_{|V^{\circ}|})\in\mathfrak{T}$ and $\nabla\mathcal{E}$ is a map from $\mathbb{R}^{|V^{\circ}|}$ to $\mathfrak{T}$. 
		
		By Proposition \ref{pro3} and Corollary \ref{coro1}, $\nabla\mathcal{E}$ is a embedding map. Hence we only need to show that the image set of $\nabla\mathcal{E}$ is $\mathfrak{T}$. Now we need to analysis the boundary of its image.
		
		Choose a point $s^*=(s_1^*,\cdots,s_{|V^{\circ}|}^*)\in \partial\mathbb{R}^{|V^{\circ}|}$, we define two subset $W_1, W_2\subset V^{\circ}$, i.e.
		$$
		W_1=\{i\in V^{\circ}~|~s_i^*=+\infty\},~W_2=\{i\in V^{\circ}~|~s_i^*=-\infty\},
		$$
		we know that $W_1\ne\emptyset$ or $W_2\ne\emptyset$.
		
		We choose a sequence $\{s^n\}_{n=1}^{+\infty}\subset \mathbb{R}^{|V^{\circ}|}$ such that $\lim_{n\to +\infty}s^n=s^*$, where $s^n=(s_1^n,\cdots,s_{|V^{\circ}|}^n)$. Then we have
		$$
		\nabla\mathcal{E}(s^n)=T(s^n)\overset{\Delta}{=}T^n,~k_i^n\overset{\Delta}{=}\exp(s_i^n),~\forall n\ge 1.
		$$
		Now we need to show that $\{T^n\}_{n=1}^{+\infty}$ converges to the boundary of $\mathfrak{T}$.
		
		If $W_1\ne\emptyset$, we know that for each $i\in W_1$, $s_i^n\to +\infty (n\to +\infty)$ and $k_i^n=\exp(s_i^n)\to +\infty ~(n\to +\infty)$. Besides, note that
		\begin{align}\label{for4}
			\lim_{n\to +\infty}\sum_{v\in W_1}T_v^n=\lim_{n\to +\infty}\sum_{v\in W_1}\sum_{f\in F_{\{v\}}}T_{v}^{f,n}=\lim_{n\to +\infty}\sum_{f\in F_{W_1}}\sum_{v\in V_f\cap W_1}T_{v}^{f,n}.
		\end{align}
		Now we need to show
		\begin{align}\label{for5}
			\lim_{n\to +\infty}\sum_{v\in V_f\cap W_1}T_{v}^{f,n}=\pi.
		\end{align}
		Since $W_1\ne\emptyset$ and $f\in F_{W_1}$, the set $ V_f\cap W_1$ has 1 or 2 or 3 elements. We can suppose that $V_f=\{r,s,t\}$. 
		
		If $V_f\cap W_1$ has 1 elements, we can suppose that $V_f\cap W_1=\{r\}$. Then we know that $k_r^n\to+\infty (n\to+\infty)$ and $k_s^n, k_t^n\nrightarrow+\infty (n\to+\infty)$. Then by (\ref{3.4}) in Lemma \ref{2.12}, we have 
		\begin{equation}\label{3.7}
			\lim_{n\to+\infty}\sum_{v\in V_f\cap W_1}T_v^{f,n}=\lim_{n\to+\infty}T_r^{f,n}=\pi.
		\end{equation}
		
		If $V_f\cap W_1$ has 2 elements, we can suppose that $V_f\cap W_1=\{r,s\}$. Then we know that $k_r^n, k_s^n\to+\infty (n\to+\infty)$ and $k_t^n\nrightarrow+\infty (n\to+\infty)$. Then by (\ref{3.5}) in Lemma \ref{2.12}, we have 
		\begin{equation}\label{3.8}
			\lim_{n\to+\infty}\sum_{v\in V_f\cap W_1}T_v^{f,n}=\lim_{n\to+\infty}T_r^{f,n}+T_s^{f,n}=\pi.
		\end{equation}
		
		If $V_f\cap W_1$ has 3 elements, we can suppose that $V_f\cap W_1=\{r,s,t\}$. Then we know that $k_r^n, k_s^n, k_t^n\to+\infty (n\to+\infty)$. Then by (\ref{3.6}) in Lemma \ref{2.12}, we have 
		\begin{equation}\label{3.9}
			\lim_{n\to+\infty}\sum_{v\in V_f\cap W_1}T_v^{f,n}=\lim_{n\to+\infty}T_r^{f,n}+T_s^{f,n}+T_t^{f,n}=\pi.
		\end{equation}
		Hence if $W_1\ne\emptyset$, by (\ref{3.7}), (\ref{3.8}) and (\ref{3.9}), we have 
		$$
		\lim_{n\to +\infty}\sum_{v\in W_1}T_v^n=\lim_{n\to +\infty}\sum_{v\in W_1}\sum_{f\in F_{\{v\}}}T_{v}^{f,n}=\lim_{n\to +\infty}\sum_{f\in F_{W_1}}\sum_{v\in V_f\cap W_1}T_{v}^{f,n}=\sum_{f\in F_{W_1}}\lim_{n\to +\infty}\sum_{v\in V_f\cap W_1}T_{v}^{f,n}=\pi|F_{W_1}|,
		$$
		i.e. $\{T^n\}_{n=1}^{+\infty}$ converges to the boundary of $\mathfrak{T}$.    
		
		If $W_2\ne\emptyset$, then for each $i\in W_2$, $s_i^n\to-\infty (n\to +\infty)$ and $k_i^n=\exp(s_i^n)\to 0 ~(n\to +\infty)$. By (\ref{3.3}) in Lemma \ref{2.12}, we have 
		$$
		\lim_{n\to+\infty}T_i^n=\lim_{n\to+\infty}\sum_{f\in F_{\{i\}}}T_i^{f,n}=0.
		$$
		i.e. $\{T^n\}_{n=1}^{+\infty}$ converges to the boundary of $\mathfrak{T}$.
		
		By Brouwer's Theorem on the Invariance of Domain and the above analysis, we know that the image set of $\nabla\mathcal{E}$ is $\mathfrak{T}$. Hence $\nabla\mathcal{E}$ is a homeomorphism from $\mathbb{R}^{|V^{\circ}|}$ to $\mathfrak{T}$. 
	\end{proof}
	
	We can construct a map $\varsigma$ from $\mathbb{R}^{|V^{\circ}|}_{>0}$ to $\mathbb{R}^{V^{\circ}|}$, i.e.
	$$
	\begin{array}{cccc}
		\varsigma: \mathbb{R}^{|V^{\circ}|}_{>0} &\longrightarrow &\mathbb{R}^{|V^{\circ}|}\\
		k=(k_1,\cdots,k_{|V^{\circ}|})&\longmapsto & s=(\ln k_1, \cdots, \ln k_{|V^{\circ}|}),\\
	\end{array}
	$$
	the map $\varsigma$ is a homeomorphism. By Theorem \ref{thm1}, we know that the map $\mathcal{E}_1:=\nabla\mathcal{E}\circ\varsigma$, i.e.
	$$
	\begin{array}{cccc}
		\mathcal{E}_1: \mathbb{R}^{|V^{\circ}|}_{>0} &\longrightarrow &\mathfrak{T}\\
		k&\longmapsto & \mathcal{E}_1(k)=\nabla\mathcal{E}(s)=T(s)\\
	\end{array}
	$$
	is a homeomorphism from $\mathbb{R}^{|V^{\circ}|}_{>0}$ to $\mathfrak{T}$. 
	
	For $\hat{k}=(\hat{k}_1,\cdots,\hat{k}_{|V^{\partial}|})\in\mathbb{R}_{>0}^{|V^{\partial}|}$ and $\hat{T}=(\hat{T}_1,\cdots,\hat{T}_{|V^{\circ}|})\in\mathbb{R}^{|V^{\circ}|}_{>0}$, a generalized circle packing $P$ is said to realize the data $(\hat{k},\hat{T})$ if its corresponding geodesic curvature is equal to $\hat{k}$ on boundary vertex set $V^{\partial}$ and corresponding total geodesic curvature is equal to $\hat{T}$ on interior vertex set $V^{\circ}$. 
 
 \begin{proof}[The proof of Theorem \ref{thm}]
From the analysis of this section, we can easily obtain Theorem \ref{thm}.
\end{proof}

	
	\section{Combinatorial Calabi flows for prescribed total geodesic curvatures }\label{s4}

	Using change of variables $s_i=\ln k_i$, we can rewrite flows (\ref{f1}) as the following two equivalent combinatorial Calabi flows.
	\begin{definition}(combinatorial Calabi flows)
		\begin{equation}\label{f2}
			\frac{ds}{dt}=-M'(T-\hat{T}).
		\end{equation}
	\end{definition}

	\begin{proof}[proof of Theorem \ref{t1}]
		We can only consider the equivalent flows (\ref{f2}). By $i\sim j$ we denote the vertices $i$ and $j$ are adjacent. Then for any vertex $i\in V^{\circ}$, we obtain
		\begin{equation}\label{e7}
			\begin{aligned}
				\frac{ds_i}{dt}&=-\sum_{j=1}^{|V^{\circ}|}\pp{T_j}{s_i}(T_j-\hat{T}_j) \\
				&=-\pp{T_i}{s_i}(T_i-\hat{T}_i)-\sum_{j\sim i}\pp{T_j}{s_i}(T_j-\hat{T}_j).\\
			\end{aligned}
		\end{equation}
		Moreover, we have
		\begin{equation}\label{e1}
			\pp{T_i}{s_i}=\sum_{f\in F_{\{i\}}}\pp{T_i^f}{s_i},~~\pp{T_j}{s_i}=\sum_{f\in F_{\{i,j\}}}\pp{T_j^f}{s_i},~~j\sim i,
		\end{equation}
		where $F_{\{i,j\}}$ is the set of all faces with $i$ and $j$ as its vertices.
		
		By the proof of Lemma 2.8 in \cite{HHSZ}, for any vertices $j\sim i$ and face $f\in F_{\{i,j\}}$, we have
		\begin{equation}
			0<\left|\pp{T_j^f}{k_i}\right|<\frac{2(k_f^2-1)}{k_f(k_f^2-1+k_i^2)}.
		\end{equation}
		Then we obtain
		\begin{equation}\label{e2}
			\left|\pp{T_j^f}{s_i}\right|=k_i\left|\pp{T_j^f}{k_i}\right|<\frac{2k_i(k_f^2-1)}{k_f(k_f^2-1+k_i^2)}=\frac{2\frac{k_i}{k_f}}{1+\frac{k_i^2}{k_f^2-1}}<\frac{2\frac{k_i}{\sqrt{k_f^2-1}}}{1+\frac{k_i^2}{k_f^2-1}}\le 1.
		\end{equation}
		By (\ref{e1}) and(\ref{e2}), we have
		\begin{equation}\label{e8}
			\left|\pp{T_j}{s_i}\right|=\left|\sum_{f\in F_{\{i,j\}}}\pp{T_j^f}{s_i}\right|\le \sum_{f\in F_{\{i,j\}}}\left|\pp{T_j^f}{s_i}\right|\le 2.
		\end{equation}
		
		For any vertex $i\in V^{\circ}$ and $f\in F_{\{i\}}$, by Gauss-Bonnet theorem, we have
		$$
		T_i^f+\sum_{j\in V_f-\{i\}}T_j^f+\text{Area}(\Omega_f)=\pi.
		$$
		Then we obtain 
		\begin{equation}\label{e3}
			\pp{T_i^f}{s_i}=-\pp{\text{Area}(\Omega_f)}{s_i}-\sum_{j\in V_f-\{i\}}\pp{T_j^f}{s_i}.
		\end{equation}
		By Theorem 4.2 in \cite{HHSZ}, we know that $\pp{\text{Area}(\Omega_f)}{s_i}=k_i\pp{\text{Area}(\Omega_f)}{k_i}$ is bounded. Then there exists a constant $C$ such that 
		\begin{equation}\label{e4}
			\left|\pp{\text{Area}(\Omega_f)}{s_i}\right|\le C.
		\end{equation}
		By (\ref{e2}), (\ref{e3}) and (\ref{e4}), we have
		\begin{equation}\label{e5}
			\left|\pp{T_i^f}{s_i}\right|\le\left|\pp{\text{Area}(\Omega_f)}{s_i}\right|+\sum_{j\in V_f-\{i\}}\left|\pp{T_j^f}{s_i}\right|\le 2+C.
		\end{equation}
		By $d_i$ we denote the number of all faces with $i$ as one of its vertices. 
		Then by (\ref{e1}) and (\ref{e5}), we have 
		\begin{equation}\label{e9}
			\left|\pp{T_i}{s_i}\right|\le\sum_{f\in F_{\{i\}}}\left|\pp{T_i^f}{s_i}\right|\le (2+C)d,
		\end{equation}
		where $d=\text{max}_{i\in V^{\circ}}d_i$.
		Moreover, by Gauss-Bonnet theorem, we have
		\begin{equation}\label{e6}
			\left|T_i\right|\le\sum_{f\in F_{\{i\}}}\left|T_i^f\right|\le\pi d,~~\forall i\in V^{\circ}.
		\end{equation}
		Then by (\ref{e7}), (\ref{e8}), (\ref{e9}) and (\ref{e6}), we obtain
		$$
		\left|\frac{ds_i}{dt}\right|\le\left|\pp{T_i}{s_i}\right|\left|T_i-\hat{T}_i\right|+\sum_{j\sim i}\left|\pp{T_j}{s_i}\right|\left|(T_j-\hat{T}_j)\right|\le(4+C)(\pi d+\left|\left|\hat{T}\right|\right|_{\infty})d<+\infty.
		$$
		Hence all $|\frac{ds_i}{dt}|$ are uniformly bounded by a constant. Then by the extension theorem of solutions in ODE theory, the solution of the combinatorial Calabi flows (\ref{f1}) exists for all time $t\ge 0$.
	\end{proof}
	
	Given $\hat{k}=(\hat{k}_1,\cdots,\hat{k}_{|V^{\partial}|})\in\mathbb{R}_{>0}^{|V^{\partial}|}$ and $\hat{T}=(\hat{T}_1,\cdots,\hat{T}_{|V^{\circ}|})\in\mathfrak{T}$, then we can define a 1-form
	$$
	\omega=\sum_{i=1}^{|V^{\circ}|}(T_i-{\hat{T}}_i)\mathrm{~d} s_i,
	$$
	it is easy to see that $\omega$ is closed.
	
	By Theorem \ref{thm1}, there exists $\hat{s}\in \mathbb{R}^{|V^{\circ}|}$ which satisfies $T(\hat{s})=\hat{T}$, then we can define the potential function
	$$
	\Theta(s)=\int_{\hat{s}}^s \omega.
	$$
	The function $\Theta$ is well-defined and does not depend on the particular choice of a piecewise smooth arc in $\mathbb{R}^{|V^{\circ}|}$ from the initial point $\hat{s}$ to $s$.
	
	\begin{proposition}
		The potential function $\Theta$ is strictly convex. 
	\end{proposition}
	\begin{proof}
		By the definition of $\Theta$, we have 
		$$
		\nabla\Theta=T-\hat{T},~~\text{Hess}~\Theta=M.
		$$
		Then by Proposition \ref{pro3}, $\text{Hess}~\Theta$ is positive definite and $\Theta$ is strictly convex.
	\end{proof}
	
	\begin{proposition}\label{p5}
		$\hat{s}$ is the unique critical point of the potential function $\Theta$.
	\end{proposition}
	\begin{proof}
		Since $\Theta$ is strictly convex and $\text{Hess}~\Theta$ is positive definite, then $\nabla\Theta$ is a embedding map. Moreover, we have $T(\hat{s})=\hat{T}$, i.e. 
		$$\nabla\Theta(\hat{s})=T(\hat{s})-\hat{T}=0,$$ 
		hence $\hat{s}$ is a critical point of $\Theta$. Since $\nabla\Theta$ is a embedding map, then $\hat{s}$ is the unique critical point of $\Theta$.
	\end{proof}
	
	\begin{proposition}\label{p6}
		Suppose $s(t)$ is a long time solution to the combinatorial Calabi flow \ref{f2}, then $s(t)$ converges to $\hat{s}$, i.e.
		$$
		\operatorname{lim}_{t\rightarrow +\infty} s(t)=\hat{s}.
		$$   
	\end{proposition}
	\begin{proof}
		Since the potential function $\Theta$ is strictly convex and $\hat{s}$ is the unique critical point of $\Theta$, then we have
		$$
		\Theta(s)\ge \Theta(\hat{s})=0, \forall s\in \mathbb{R}^{|V^{\circ}|}.
		$$
		We can define a function $\Phi(t)=\Theta(s(t))$, since $M$ is positive definite, then we obtain
		$$
		\begin{aligned}
			\frac{d\Phi}{dt}=\nabla\Theta(s(t))^{\prime} \cdot \frac{ds}{dt}=-(T(s(t))-\hat{T})^{\prime}M'(T(s(t))-\hat{T})\le 0,
		\end{aligned}
		$$
		then $\Phi(t)$ is decreasing on [0,+$\infty$). Since $\Phi(t)=\Theta(s(t))\ge 0$ and $\Phi(t)$ is a decreasing function, then $\Phi(t)$ converges, i.e. $\operatorname{lim}_{t\rightarrow +\infty} \Phi(t)$ exists and 
		\begin{equation}\label{e10}
			0\le \Phi(t)=\Theta(s(t))\le \Phi(0)=\Theta(s(0)),~~\forall t\in [0,+\infty).
		\end{equation}
		Since $\Theta$ is strictly convex and has a critical point $\hat{s}$, then $\Theta$ is a proper map. By (\ref{e10}) and $\Theta$ is a proper map, we have $\{s(t)\}_{t\ge 0}\subset\Theta^{-1}[0,\Theta(s(0))]$ and $\Theta^{-1}[0,\Theta(s(0))]$ is compact, hence $\{s(t)\}_{t\ge 0}$ is compactly supported in $\mathbb{R}^{|V^{\circ}|}$.
		
		Using the mean value theorem for the function $\Phi(t)$, we have 
		\begin{align}
			\Phi(n+1)-\Phi(n)=\left.\frac{d\Phi}{dt}\right |_{t=\xi_n}, ~~\forall n\in \mathbb{N},\label{mean value}
		\end{align}
		for some $\xi_n\in[n,n+1]$. Since $\operatorname{lim}_{t\rightarrow +\infty} \Phi(t)$ exists, then we have 
		\begin{equation}\label{e11}
			0=\operatorname{lim}_{n\rightarrow +\infty} \Phi(n+1)-\operatorname{lim}_{n\rightarrow +\infty}\Phi(n)=\operatorname{lim}_{n\rightarrow +\infty} \left.\frac{d\Phi}{dt}\right |_{t=\xi_n}.
		\end{equation}
		Since $\{s(t)\}_{t\ge 0}$ is compactly supported in $\mathbb{R}^{|V^{\circ}|}$, then $\{s(\xi_n)\}_{n\ge 0}$ has a convergent subsequence, we still use $\{s(\xi_n)\}_{n\ge 0}$ to denote this convergent subsequence for simplicity. Then there exists a point $s^*\in \mathbb{R}^{|V^{\circ}|}$ such that
		$$
		\operatorname{lim}_{n\rightarrow +\infty} s(\xi_n)=s^*.
		$$
		Then by (\ref{e11}), we have
		$$
		\begin{aligned}
			\operatorname{lim}_{n\rightarrow +\infty}\left.\frac{d\Phi}{dt}\right |_{t=\xi_n}&=
			\operatorname{lim}_{n\rightarrow +\infty}-(T(s(\xi_n))-\hat{T})^{\prime}M'(s(\xi_n))(T(s(\xi_n))-\hat{T})\\
			&=-(T(s^*)-\hat{T})^{\prime}M'(s^*)(T(s^*)-\hat{T})\\
			&=0.
		\end{aligned}
		$$
		Since $M'(s^*)$ is positive definite, then $T(s^*)-\hat{T}=0$. Hence we have
		$$
		\nabla\Theta(s^*)=T(s^*)-\hat{T}=0,
		$$
		which implies that $s^*$ is a critical point of $\Theta$. By Proposition \ref{p5}, $\hat{s}$ is the unique critical point of the potential function $\Theta$, then we have $s^*=\hat{s}$ and 
		$$
		\operatorname{lim}_{n\rightarrow +\infty}s(\xi_n)=\hat{s}.
		$$
		Then we obtain
		$$
		\operatorname{lim}_{n\rightarrow +\infty}\Phi(\xi_n)=\operatorname{lim}_{n\rightarrow +\infty}\Theta(s(\xi_n))=\Theta(\hat{s})=0,$$ 
		which implies 
		\begin{equation}\label{e12}
			\operatorname{lim}_{t\rightarrow +\infty} \Phi(t)=0.
		\end{equation}
		If $s(t)$ does not converge to $\hat{s}$, then there exists $\delta>0$ and a sequence $\{t_n\}_{n\ge 1}$ such that $\operatorname{lim}_{n\rightarrow +\infty} t_n=+\infty$ and 
		$$
		|s(t_n)-\hat{s}|>\delta,~~\forall n\ge 1,
		$$
		then we have $\{s(t_n)\}_{n\ge 1}\subset\mathbb{R}^{|V^{\circ}|}\setminus B(\hat{s},\delta)$. Since $\{s(t_n)\}_{n\ge 1}\subset\Theta^{-1}[0,\Theta(s(0))]\cap(\mathbb{R}^{|V^{\circ}|}\setminus B(\hat{s},\delta))$ and there exists a positive constant $C$ such that 
		$$
		\Theta(s)\ge C>0,~~\forall s\in \Theta^{-1}[0,\Theta(s(0))]\cap(\mathbb{R}^{|V^{\circ}|}\setminus B(\hat{s},\delta)).
		$$
		Then we have
		$$
		\operatorname{lim}_{n\rightarrow +\infty} \Phi(t_n)=\operatorname{lim}_{n\rightarrow +\infty}\Theta(s(t_n))\ge C>0,
		$$
		which contradicts (\ref{e12}). This completes the proof.
	\end{proof}
	
	\begin{proof}[The proof of Theorem \ref{Calabi flow}]
		We can only consider the equivalent flows (\ref{f2}). By Theorem \ref{t1}, there exists the unique long time solution $s(t)$ to the combinatorial Calabi flow (\ref{f2}). Since $\hat{T}=(\hat{T}_1,\cdots,\hat{T}_{|V^{\circ}|})\in\mathfrak{T}$, then by Theorem \ref{thm1}, there exists a point $\hat{s}\in\mathbb{R}^{|V^{\circ}|}$ such that 
		$$
		\nabla\mathcal{E}(\hat{s})=T(\hat{s})=\hat{T}.
		$$
		By Proposition \ref{p6}, we have 
		$$
		\operatorname{lim}_{t\rightarrow +\infty} s(t)=\hat{s},
		$$
		which implies the solution to the combinatorial Calabi flow (\ref{f1}) converges to the generalized circle packing packing on $S_{g,n}$ which realizes $(\hat{k},\hat{T})$.     
	\end{proof}

	\section{Discrete Schwartz-Pick Lemma for generalized circle packing in hyperbolic background geometry}\label{sd}
	
	Given two sets of prescribed geodesic curvatures $\hat{k},\hat{k}^*\in\mathbb{R}_{>0}^{|V^{\partial}|}$ on boundary vertex set $V^{\partial}$ and prescribed total geodesic curvatures $\hat{T}\in\mathfrak{T}$ on interior vertex set $V^{\circ}$, then by Theorem \ref{thm}, there exists generalized circle packings $P,P^*$ on $S_{g,n}$ which realize $(\hat{k},\hat{T})$, $(\hat{k}^*,\hat{T})$, respectively. 
	
	By $k(P(v)), k(P^*(v))$ we denote the geodesic curvature on the vertex $v$ in the generalized circle packings $P,P^*$, respectively. Then we have the following maximum principle for generalized circle packing in hyperbolic background geometry.
	\begin{proof}[The proof of Theorem \ref{thm5}]
		For simplicity, let $k(P(v))=k_v$ and $k(P^*(v))=k_v^*$. We prove (a) of Theorem \ref{thm5}. If not, there exists a interior vertex \( v_0 \) such that
		
		\begin{equation}\label{4.1}
			\frac{k_{v_0}^*}{k_{v_0}} = \max_{v \in V}\left\{\frac{k_v^*}{k_v}\right\} > 1,~~
			\frac{k^*_{v_0}}{k_{v_0}} \geq \frac{k^*_v}{k_v}, ~\forall v \in V.
		\end{equation}
		\begin{figure}[htbp]
			\centering
			\includegraphics[scale=0.6]{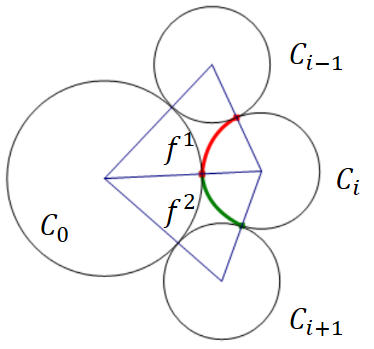}
			\captionof{figure}{\small The faces $f^1$ and $f^2$ with common edge $v_0v_i$. }
			\label{8}
		\end{figure}
		For simplicity, let \( V_{v_0} = \{v_1, v_2, \dots, v_n\} \) denote the set of vertices connected to \( v_0 \) with the corresponding circles \( C_0, C_1, \dots, C_n \) in the generalized circle packing, arranged clockwise such that \( C_{n+1} = C_1 \). 
		
		We can define $k^{(*)}_i := k^{(*)}_{v_i}$, $s^{(*)}_i := \ln k^{(*)}_i$ and \(T_i := T_{v_i} \), \( i = 0, 1, 2, \dots, n \). By (\ref{4.1}), we have
		$$
		\ln\frac{k_0^*}{k_0}>0,~~\ln\frac{k^*_0}{k_0}\geq\ln\frac{k^*_i}{k_i},~i=1,\cdots,n,
		$$
		which implies
		\begin{equation}\label{4.7}
			s^*_0-s_0>0,~~s^*_0-s_0\geq s^*_i-s_i,~i=1,\cdots,n.
		\end{equation}
		
		For $t \in [0,1]$, we can define functions $s_i(t):= (1-t)s_i +ts^*_i$, \( i=0, 1, 2, \dots, n \). Then we have  
		$$
		s_i(0)=s_i,~~s_i(1)=s_i^*,~i=0, 1, 2, \dots, n. 
		$$
		
		We can define functions $s(t):= (s_0(t), s_1(t), \cdots, s_n(t))$, $k(t):= (e^{s_0(t)}, e^{s_1(t)}, \cdots, e^{s_n(t)})$ and $T_0(t):=T_0(s(t))$. By the assumption of Theorem \ref{thm5}, we have $T_0(0)=T_0(1)= \hat T_0$. Then by mean value theorem, there exist a point $\xi \in [0,1]$ such that 
		\begin{equation}\label{4.9}
			0= T_0(1)-T_0(0)=T'_0(\xi)= \sum_{i=0}^{n}\left.\frac{\partial T_0}{\partial s_i}\right|_{t=\xi} \cdot \frac{d s_i}{d t} = \sum_{i=0}^{n}\left.\frac{\partial T_0}{\partial s_i}\right|_{t=\xi}(s_i^*-s_i).
		\end{equation}
		
		For the total geodesic curvature $T_0$ on the vertex $v_0$, we have $T_0= \sum_{f\in F_0} T_0^f$, where $F_0$ is the set of all faces with $v_0$ as one of its vertices. Choose a vertex $v_i\in V_{v_0}$, we can assume $f^1$, $f^2$ are two faces which are adjacent with edge $v_0v_i$ as shown in Figure \ref{8}.

		Then for each face $f\in F_0\setminus\{f^1, f^2\}$, we have $\pp{T_0^f}{s_i}=0$. Hence we obtain
		\begin{equation}\label{4.3}
			\frac{\partial T_0}{\partial s_i}=\sum_{f\in F_0} \pp{T_0^f}{s_i}=\frac{\partial T^{f^1}_0}{\partial s_i}+ \frac{\partial T^{f^2}_0}{\partial s_i}=\frac{\partial T^{f^1}_i}{\partial s_0}+\frac{\partial T^{f^2}_i}{\partial s_0},
		\end{equation}
		where the $T^{f^1}_i$ and $T^{f^2}_i$ are the total geodesic curvatures of the red and green arcs, respectively, as shown in Figure \ref{8}. Then we obtain 
		\begin{equation}\label{4.2}
			\frac{\partial T_0}{\partial s_i}=\sum_{f\in F_0} \pp{T_0^f}{s_i}=\sum_{f\in F_{0,i}}\pp{T_i^f}{s_0},~i=1,\cdots,n, 
		\end{equation}
		where $F_{0,i}$ is the set of all faces with edge $v_0v_i$. By (\ref{4.2}), we have 
		\begin{equation}\label{4.5}
			\sum_{i=0}^n \frac{\partial T_0}{\partial s_i}=\frac{\partial T_0}{\partial s_0}+\sum_{i=1}^n \frac{\partial T_0}{\partial s_i}=\sum_{f \in F_0} \frac{\partial T_0^f}{\partial s_0}+\sum_{i=1}^n \sum_{f \in F_{0, i}} \frac{\partial T_i^f}{\partial s_0}=\sum_{f \in F_0} \sum_{j \in V_f} \frac{\partial T_j^f}{\partial s_0}
		\end{equation}
		where $T_j^f$ are the total geodesic curvatures of the blue arcs as shown in Figure \ref{10}.
		
		\begin{figure}[htbp]
			\centering
			\includegraphics[scale=0.4]{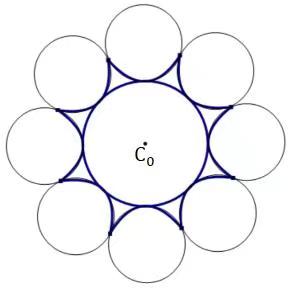}
			\captionof{figure}{\small $T_j^f$'s are the total geodesic curvatures of the blue arcs.  }
			\label{10}
		\end{figure}

		By Lemma 2.11 in \cite{BHS}, we obtain 
		\begin{equation}\label{4.4}
			\sum_{j\in V_f}\pp{T_j^f}{s_0}>0,~\forall f\in F_0,~~\pp{T_i^f}{s_0}<0,~\forall f\in F_{0,i}.
		\end{equation}
		By (\ref{4.2}), (\ref{4.5}) and (\ref{4.4}), we have
		$$
		\sum_{i=0}^n\frac{\partial T_0}{\partial s_i}=\sum_{f\in F_0}\sum_{j\in V_f}\pp{T_j^f}{s_0}>0,~~\frac{\partial T_0}{\partial s_i}=\sum_{f\in F_{0,i}}\pp{T_i^f}{s_0}<0,~i=1,\cdots,n.
		$$
		Then we obtain
		\begin{equation}\label{4.6}
			\sum_{i=0}^n\left.\frac{\partial T_0}{\partial s_i}\right |_{t=\xi}>0,~~\left.\frac{\partial T_0}{\partial s_i}\right |_{t=\xi}<0,~i=1,\cdots,n.
		\end{equation}
		
		For each $1\leq i\leq n$, if $s_i^*-s_i \geq 0$, then by (\ref{4.7}), we have $s_0^*-s_0 \geq s_i^*-s_i \geq 0$. By (\ref{4.6}), we obtain
		$$
		\left.\frac{\partial T_0}{\partial s_i}\right |_{t=\xi}(s_0^*-s_0)\leq \left.\frac{\partial T_0}{\partial s_i}\right |_{t=\xi}(s_i^*-s_i).
		$$
		If $s_i^*-s_i<0$, then by (\ref{4.7}), we have $s_0^*-s_0>0>s_i^*-s_i$. By (\ref{4.6}), we obtain 
		$$
		\left.\frac{\partial T_0}{\partial s_i}\right |_{t=\xi}(s_0^*-s_0)<0<\left.\frac{\partial T_0}{\partial s_i}\right |_{t=\xi}(s_i^*-s_i).
		$$
		Then we have 
		\begin{equation}\label{4.8}
			\left.\frac{\partial T_0}{\partial s_i}\right |_{t=\xi}(s_0^*-s_0)\leq \left.\frac{\partial T_0}{\partial s_i}\right |_{t=\xi}(s_i^*-s_i),~i=1,\cdots,n.  
		\end{equation}
		By (\ref{4.8}), we obtain
		\begin{equation}\label{4.10}
			\begin{aligned}
				\sum_{i=0}^n\left.\frac{\partial T_0}{\partial s_i}\right |_{t=\xi}(s_i^*-s_i)&=\left.\frac{\partial T_0}{\partial s_0}\right |_{t=\xi}(s_0^*-s_0)+\sum_{i=1}^n\left.\frac{\partial T_0}{\partial s_i}\right |_{t=\xi}(s_i^*-s_i) \\
				&\geq \left.\frac{\partial T_0}{\partial s_0}\right |_{t=\xi}(s_0^*-s_0)+\sum_{i=1}^n\left.\frac{\partial T_0}{\partial s_i}\right |_{t=\xi}(s_0^*-s_0)\\
				&=\left (\sum_{i=0}^n\left.\frac{\partial T_0}{\partial s_i}\right |_{t=\xi}\right )(s_0^*-s_0).
			\end{aligned}
		\end{equation}
		By (\ref{4.7}), (\ref{4.6}) and (\ref{4.10}), we have
		$$
		\sum_{i=0}^n\left.\frac{\partial T_0}{\partial s_i}\right |_{t=\xi}(s_i^*-s_i)\geq \left (\sum_{i=0}^n\left.\frac{\partial T_0}{\partial s_i}\right |_{t=\xi}\right )(s_0^*-s_0)>0,
		$$
		which contradicts (\ref{4.9}). This completes the proof of $(a)$. 
		
		We prove (b) of Theorem \ref{thm5}. If not, there exists a interior vertex $v_0$ such that $\frac{k^*_{v_0}}{k_{v_0}} > 1$. Then we have
		$$
		\max_{v \in V}\left\{\frac{k^*_v}{k_v}\right\} \geq \frac{k^*_{v_0}}{k_{v_0}}> 1.
		$$
		By (a) of Theorem \ref{thm5}, there exists a boundary vertex $v'_0$ such that 
		\[
		\frac{k^*_{v'_0}}{k_{v'_0}} = \max_{v \in V}\left\{\frac{k^*_v}{k_v}\right\} > 1,
		\]
		which contradicts assumption of $(b)$. This completes the proof of $(b)$.
	\end{proof}
	
	By $\operatorname{Area}(\Omega_f), \operatorname{Area}(\Omega^*_f)$ we denote the region enclosed by sub-arcs in generalized circle packings $P, P^*$ on the face $f$, respectively. By $C_v, C_v^*$ we denote the generalized circles on the vertex $v$ in generalized circle packings $P, P^*$, respectively. By $l_v^f, l_v^{f*}$ we denote the lengths of sub-arcs $C_v^f, C_v^{f*}$ of generalized circles $C_v, C_v^*$ in dual circles $C_f, C_f^*$, respectively, as shown in Figure \ref{9}. Then we have the following discrete Schwartz-Pick Lemma for generalized circle packing in hyperbolic background geometry.
	
	\begin{figure}[htbp]
		\centering
		\includegraphics[scale=0.25]{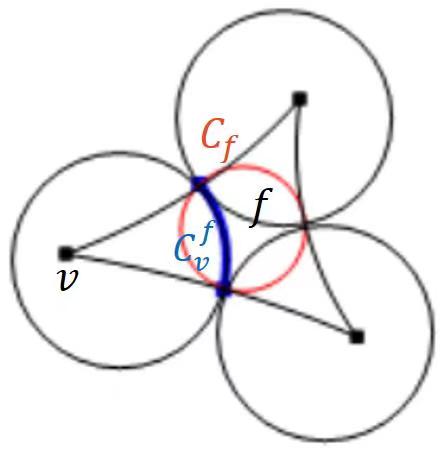}
		\captionof{figure}{\small The sub-arc $C_v^f$ of generalized circle $C_v$.}
		\label{9}
	\end{figure}

	\begin{proof}[The proof of Theorem \ref{thm6}]
		We prove the first part of Theorem \ref{thm6}. Since $k(P^*(v))\le k(P(v))$ is true for each boundary vertex, by $(b)$ in Theorem \ref{thm5}, we have $k_v^*\le k_v$ for each vertex $v$. By Lemma 2.11 in \cite{BHS}, we have
		$$
		\frac{\partial \operatorname{Area}(\Omega_f)}{\partial k_v}<0,~\forall v\in V_f.
		$$
		Then we obtain $\operatorname{Area}(\Omega^*_f)\ge\operatorname{Area}(\Omega_f)$ for every face $f$. This completes the proof of $(a)$ of Theorem \ref{thm6}.
		
		We prove the second part of Theorem \ref{thm6}. The length $l_v^f$ is a function with respect to $k_f$ and $k_v$, where $k_f$ is the geodesic curvature of $C_f$. By Lemma 2.7 in \cite{BHS}, we have $\frac{\partial l_v^f}{\partial k_f}= \frac{2}{1-k_v^2-k_f^2}<0$ and 
		$$
		\frac{\partial l_v^f}{\partial k_v}= \begin{cases}-\frac{2k_v}{(1-k_v^2)^{\frac{3}{2}}}\left (\frac{k_f \sqrt{1-k_v^2}}{k_v^2+k_f^2-1}-\operatorname{arctanh}\frac{\sqrt{1-k_v^2}}{k_f} \right ), & \text { if } 0<k_v<1, \\-\frac{4}{3}\frac{1}{k_f^3}, & \text { if } k_v=1, \\ \frac{2k_v}{(k_v^2-1)^{\frac{3}{2}}}\left (\frac{k_f \sqrt{k_v^2-1}}{k_v^2+k_f^2-1}-\arctan\frac{\sqrt{k_v^2-1}}{k_f} \right ), & \text { if } k_v>1,\end{cases}
		$$
		where $\frac{\partial l_v^f}{\partial k_v}$ is a continuous function.
		
		We can define a function 
		$$
		f(x):= \frac{x}{1-x^2}-\operatorname{arctanh}x,
		$$
		which is a strictly increasing function on $[0,1)$. Then if $0<k_v<1$, we have
		$$
		\frac{k_f \sqrt{1-k_v^2}}{k_v^2+k_f^2-1}-\operatorname{arctanh}\frac{\sqrt{1-k_v^2}}{k_f}=f\left (\frac{\sqrt{1-k_v^2}}{k_f}\right )>f(0)=0,
		$$
		which implies 
		$$
		\frac{\partial l_v^f}{\partial k_v}=-\frac{2k_v}{(1-k_v^2)^{\frac{3}{2}}}\left (\frac{k_f \sqrt{1-k_v^2}}{k_v^2+k_f^2-1}-\operatorname{arctanh}\frac{\sqrt{1-k_v^2}}{k_f} \right )<0.
		$$
		
		We can define a function 
		$$
		g(x):= \frac{x}{1+x^2}-\arctan x,
		$$
		which is a strictly decreasing function on $(0,+\infty)$. Then if $k_v>1$, we have 
		$$
		\frac{k_f \sqrt{k_v^2-1}}{k_v^2+k_f^2-1}-\arctan\frac{\sqrt{k_v^2-1}}{k_f} = g\left (\frac{\sqrt{k_v^2-1}}{k_f}\right )<g(0)=0,
		$$
		which implies
		$$
		\frac{\partial l_v^f}{\partial k_v}=\frac{2k_v}{(k_v^2-1)^{\frac{3}{2}}}\left (\frac{k_f \sqrt{k_v^2-1}}{k_v^2+k_f^2-1}-\arctan\frac{\sqrt{k_v^2-1}}{k_f} \right )<0.
		$$
		Hence we have 
		\begin{equation}\label{4.12}
			\frac{\partial l_v^f}{\partial k_v}<0,~ \frac{\partial l_v^f}{\partial k_f}<0.
		\end{equation}
		Let $v_i$, $v_j$, $v_l$ be the vertices of face $f$, then by Corollary 4.2 in \cite{cr}, we have 
		\begin{equation}\label{4.11}
			k_f^2= k_{v_i}k_{v_j}+k_{v_j}k_{v_l}+k_{v_i}k_{v_l}+1.
		\end{equation}
		Since $k_v^*\le k_v$ for each vertex $v$, by (\ref{4.11}), we have $k_f^*\le k_f$ for each face $f$. Then since $k_v^*\le k_v$ for each vertex $v$ and $k_f^*\le k_f$ for each face $f$, by (\ref{4.12}), we have $l_v^{f*}\ge l_v^f$ for every face $f$ and every vertex of $f$. This completes the proof of $(b)$ of Theorem \ref{thm6}.
		
		\begin{figure}[htbp]
			\centering
			\includegraphics[scale=0.3]{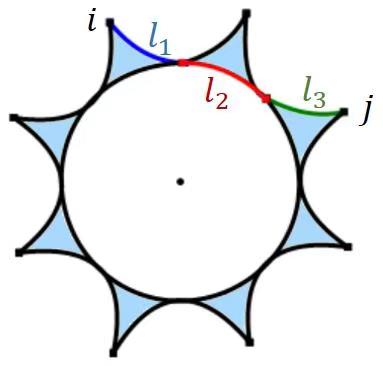}
			\captionof{figure}{\small The shortest distance $d_{ij}$ between tangent points $i$ and $j$ is the sum of $l_1$, $l_2$ and $l_3$.}
			\label{12}
		\end{figure}
		
		We prove the third part of Theorem \ref{thm6}. For the surfaces $\hat{S}$ and $\hat{S}^{*}$, since the shortest distance between any two points on their boundary is the sum of the lengths of some arcs as shown in Figure \ref{12}, then by $(b)$ of Theorem \ref{thm6}, we have 
		$$
		d_{ij}^*=\sum l_v^{f*}\ge\sum l_v^f=d_{ij}. 
		$$
		This completes the proof of $(c)$ of Theorem \ref{thm6}.
	\end{proof}

	\section{Acknowledgments}
Guangming Hu is supported by NSF of China (No.12101275). Ziping Lei is supported by NSF of China (No. 12122119). 
 Yanlin Li is supported by NSF of China (No. 12101168).
 Hao Yu is supported by NSF of China (No. 12341102).

	\Addresses

\end{document}